\documentclass[12pt]{amsart}
\usepackage{amsmath, amsfonts, amsthm, amssymb}
\usepackage{graphicx}
\usepackage{float}
\usepackage{color}
\usepackage{tikz}
\usepackage{xcolor}
\usepackage{pgfplots}
\usepackage{caption}
\usepackage{subcaption}
\usepackage{verbatim}

\numberwithin{equation}{section}

\newtheorem{thm}{Theorem}[section]
\newtheorem{lma}[thm]{Lemma}
\newtheorem{cor}[thm]{Corollary}
\newtheorem{defn}[thm]{Definition}

\newtheorem{prop}[thm]{Proposition}

\renewcommand{\ge}{\geqslant}
\renewcommand{\le}{\leqslant}
\renewcommand{\geq}{\geqslant}
\renewcommand{\leq}{\leqslant}
\renewcommand{\H}{\text{H}}
\renewcommand{\P}{\text{P}}
\renewcommand{\l}{\overline{\dim}_{\text{loc}}(x,\mu)}
\renewcommand{\r}{\overline{\dim}_{\textup{reg}} \mu}
\renewcommand{\a}{\dim_{\textup{A}}}
\renewcommand{\epsilon}{\varepsilon}

\newcommand{\Height}{8}
\newcommand{\Width}{8}

\allowdisplaybreaks

\begin{document}

\title[Upper regularity dimensions]{On the upper regularity dimensions of measures}

\author[J. M. Fraser]{Jonathan M. Fraser}
\address{Jonathan M. Fraser\\
School of Mathematics \& Statistics\\University of St Andrews\\ St Andrews\\ KY16 9SS\\ UK  }
\curraddr{}
\email{jmf32@st-andrews.ac.uk}

\author[D. C. Howroyd]{ Douglas C. Howroyd}
\address{Douglas C. Howroyd\\
School of Mathematics \& Statistics\\University of St Andrews\\ St Andrews\\ KY16 9SS\\ UK  }
\curraddr{}
\email{dch8@st-andrews.ac.uk}



\begin{abstract}
We study the \emph{upper regularity dimension} which describes the extremal local scaling behaviour of a measure and effectively quantifies the notion of \emph{doubling}. We conduct a thorough study of the upper regularity dimension, including its relationship with other  concepts  such as the Assouad dimension, the upper local dimension,  the $L^q$-spectrum and weak tangent measures.   We also compute the upper regularity dimension explicitly in a number of important contexts including self-similar measures, self-affine measures, and measures on sequences. 
\end{abstract}

\maketitle

\section{Upper regularity and upper local dimensions} \label{intro}\label{dimension}

The \emph{upper regularity dimension} of a  measure was introduced in \cite{anti1,anti2} and gives a quantifiable description of the extremal local scaling of the measure.  In a concrete sense, which we will explain, the upper regularity dimension of a measure can be viewed as the `Assouad dimension' of a measure.  The Assouad dimension of \emph{sets} has been gaining a lot of attention in the literature on fractal geometry recently and so it is natural to consider the analogous concept for measures and to place this interesting concept in a wider mathematical context. We investigate its relationship with more familiar concepts such as the local dimensions, the $L^q$-spectrum, doubling properties, and weak tangents. We also compute the upper regularity dimensions explicitly in a number of concrete settings, such as self-similar measures satisfying the strong separation property, self-affine measures supported on carpets and sponges satisfying the very strong separation property and measures supported on convergent sequences.  These examples will exhibit several different types of behaviour and will demonstrate the sharpness of our general results.

Let $\mu$ be a locally finite Borel measure on a metric space $X$ and write $\text{supp}(\mu)$ for the support of $\mu$. The upper regularity dimension of $\mu$ is defined by 
\begin{multline*} 
\overline{\dim}_{\text{reg}} \mu = \inf \Bigg\{ s \geq 0 \, \,  : \,  \text{ there exists a  constant }C  > 0\text{  such that, for all  $0< r< R $} \\ \text{  and all $x \in \text{supp} (\mu)$, we have }  \ \  \frac{\mu(B(x,R))}{\mu(B(x,r))} \leq C\left(\frac{R}{r}\right)^{s} \Bigg\}
\end{multline*}
where $B(x,R)$ is the open ball of radius $R$ and centre $x$.   We adopt the convention that $\inf \emptyset = + \infty$. This can indeed be thought of as the `Assouad dimension' of a measure since the Assouad dimension, which is a purely metric notion describing the extremal scaling behaviour of a set in a metric space, can be expressed in terms of upper regularity dimensions of measures  supported on the set, see \cite{luksak, konyagin}. In particular, for a complete metric space $F$
\[
\dim_{\text{A}} F = \inf \left\{ \overline{\dim}_{\text{reg}} \mu \,  \colon \, \mu \text{ is a Borel probability measure supported on } F\right\}.
\]
Therefore the upper regularity dimension of a measure is always bounded below by the  the Assouad dimension of its support and in some cases the upper regularity dimension will actually provide the exact Assouad dimension, for example \cite[Theorem 2.3]{fraser-howroyd}.  Note that  a non-complete metric space is dense in its completion and the Assouad dimension is preserved under taking closure, so we may assume without loss of generality that our spaces are complete.

There is also a fairly straightforward relationship with the upper local dimensions of a measure. The upper local dimension of $\mu$ at $x \in \text{supp}(\mu)$ is defined by
\[
\overline{\dim}_{\text{loc}}(x,\mu)=\limsup_{r\rightarrow 0} \frac{\log \mu(B(x,r))}{\log r}.
\]
The lower local dimension $\underline{\dim}_{\text{loc}}(x,\mu)$ is defined in a similar way, replacing $\limsup$ with $\liminf$.  These local dimensions clearly depend on the point $x$, but naturally give rise to dimensions depending only on $\mu$.  For example, the lower Hausdorff dimension of $\mu$ is
\[
\underline{\dim}_\H \mu=\text{ess}  \inf \left\{  \underline{\dim}_{\text{loc}}(x,\mu)  \ : \  x\in \text{supp}(\mu) \right\} 
\]
and the upper packing dimension is
\[
\overline{\dim}_\P \mu=\text{ess}  \sup \left\{ \overline{\dim}_{\text{loc}}(x,\mu) \ : \  x\in \text{supp}(\mu) \right\}.
\]
A measure $\mu$ is called exact dimensional when $\underline{\dim}_\H \mu=\overline{\dim}_\P \mu$.   One immediately obtains that for any locally finite Borel measure $\mu$
\[
 \r \geq \sup \left\{ \overline{\dim}_{\text{loc}}(x,\mu) \ : \  x\in \text{supp}(\mu) \right\},
\]
which shows that the upper regularity dimension is sensitive to large upper local dimension even at a single point.  In particular, let $x \in \text{supp}(\mu)$ and let $s > \r$.  Then for small enough $r < R$ we have
\[
\frac{\mu(B(x,R))}{\mu(B(x,r))} \leq C\left(\frac{R}{r}\right)^{s}
\]
for some constant $C>1$.  Fixing $R$ and letting $r \to 0$ we obtain
\[
\overline{\dim}_{\text{loc}}(x,\mu)=\limsup_{r\rightarrow 0} \frac{\log \mu(B(x,r))}{\log r} \leq  \limsup_{r\rightarrow 0} \frac{\log \mu(B(x,R)) (r/R)^s/C }{\log r}  = s
\]
as required.  This lower bound, combined with the Assouad dimension, gives a concrete and sometimes sharp lower bound on the upper regularity dimension.   For example, we show that for any self-similar measure $\mu$ satisfying the strong separation condition we have $ \r = \sup \left\{ \overline{\dim}_{\text{loc}}(x,\mu) \ : \  x\in \text{supp}(\mu) \right\}$.  However, for self-affine measures $\mu$, one may have
\[
\r >  \sup \left\{ \overline{\dim}_{\text{loc}}(x,\mu) \ : \  x\in \text{supp}(\mu) \right\}.
\]
Ahlfors-David $s$-regular measures are even more regular than exact dimensional measures in that the ratio $\mu(B(x,R))/R^s$ is uniformly bounded away from 0 and $+\infty$ where $s >0$ is the `dimension'.  An Ahlfors-David $s$-regular measure is clearly exact dimensional (with exact dimension equal to $s$) and even satisfies $\r = s$.

Another simple, but important, initial observation is that a measure is doubling if and only if it has finite upper regularity dimension.  Recall that  a  (Borel) measure is doubling if there exists a constant $0<C< \infty$ such that for all $R>0$ and $x\in \text{supp}(\mu)$ one has
\[
\frac{\mu(B(x,R))}{\mu(B(x,R/2))}\leq C.
\]
Similarly, a metric space is doubling if there exists a constant $C$ such that any ball $B(x,R)$ can be covered by fewer than $C$ balls of radius $R/2$.  It is a well-known and fundamental fact in metric geometry that a metric space is doubling if and only if it has finite Assouad dimension, see for example \cite[Lemma 9.4]{robinson}.  It turns out that the upper regularity dimension provides the natural measure theoretic analogue of this fact: a measure $\mu$ is doubling if and only if $\r <\infty$.  This follows from \cite[Lemma 3.2]{kaenmakinew}, but we include our own proof in Section \ref{doublingproof} for completeness.

\section{Results}   \label{results}

\subsection{Lower bounds and weak tangents}

We have already noted that the Assouad dimension of the support and the supremum of the upper local dimensions give elementary lower bounds for the upper regularity dimension.  We begin by refining this observation by further relating the upper regularity dimension to the (lower) $L^q$-spectrum.  Let $\mu$ be a compactly supported Borel probability measure on $\mathbb{R}^d$.  Given $q \in \mathbb{R}$ and $r>0$ we let
\[
M_r^q(\mu) = \sup  \, \left \{ \sum_{i=1}^\infty  \mu(B(x_i,r))^q \ : \ | x_i-x_j | > 2r \text{ for $i \neq j$} \right\}
\]
be the multifractal packing function, see \cite{olsenformalism}.  The  (lower) $L^q$-spectrum of $\mu$ is then given by
\[
\underline{\tau}(q)= \liminf_{r\rightarrow 0} \frac{\log M_r^q(\mu) }{\log r}, \qquad (q\in \mathbb{R}).
\]
The $L^q$-spectrum gives a description of the global fluctuations of the measure and is a key tool in multifractal analysis. For example, the Legendre transform of the $L^q$-spectrum is an upper bound for the multifractal spectrum of $\mu$ and for measures satisfying the multifractal formalism the Legendre transform of the $L^q$-spectrum is precisely the multifractal spectrum.  As such, one can bound the supremum of the upper local dimensions by the `top of the spectrum', defined by
\[
T(\mu)=\sup \left\{ s \ge 0 \colon \underline{\tau}(q) < sq \, \,, \forall q<0 \right\}
\]
 which is the gradient of the asymptote to $\underline{\tau}(q)$ as $q \rightarrow - \infty$.  We prove that the upper regularity dimension is bounded below by the top of the spectrum.

\begin{thm} \label{relationships}
Given a compactly supported Borel probability measure $\mu$ on $\mathbb{R}^d$, we have
\[
\begin{array}{ccccccccccc}
                                                         &&                  \sup_ {x\in \mathrm{supp}(\mu)}  \overline{\dim}_{\mathrm{loc}}(x,\mu)   \quad  \leq    \quad   T(\mu)                        & & &    \\
&                       \rotatebox[origin=c]{45}{$\leq$}          & &              \rotatebox[origin=c]{315}{$\leq$} & &   \\
 \overline{\dim}_\text{\emph{P}} \mu                                            & &        &&                  & \r.\\
&                \rotatebox[origin=c]{315}{$\leq$}              &&           \rotatebox[origin=c]{45}{$\leq$} & &  \\
                         &&                                            \dim_\text{\emph{P}} \mathrm{supp}(\mu)      \quad  \leq   \quad  \dim_\text{\emph{A}} \mathrm{supp}(\mu)                           & & &  
\end{array}
\]
\end{thm}

Most of the inequalities in the above theorem are well-known or have already been established in the introduction.  The only two remaining are the  inequalities involving $T(\mu)$ and we prove these in Section \ref{spectrumproof}. We take some inspiration from \cite{fraser-jordan} where the $L^q$-spectrum was related to the infimum of the lower local dimensions.

One of the most powerful tools for providing lower bounds for the Assouad dimension of a set is the well-known result of Mackay and Tyson concerning weak tangents \cite[Proposition 6.1.5]{mackaytyson}.  In particular, the Assouad dimension of any weak tangent to a set cannot exceed the Assouad dimension of the set itself.  Given that the upper regularity dimension is the `Assouad dimension of a measure', it is natural to consider the analogous result in the setting of measures, where we replace convergence in the Hausdorff metric with weak convergence.  More precisely, we say that a Borel measure  $\hat{\mu}$ on  $ \mathbb{R}^d$ is a \emph{weak tangent measure} to  a locally finite Borel measure $\mu$ on $\mathbb{R}^d$ if there exists a sequence of similarity maps $T_k$ on  $\mathbb{R}^d$ and a sequence of positive renormalising numbers $p_k$ such that
\[
p_k \mu \circ T^{-1}_k  \rightharpoonup \hat{\mu}
\]
where $\rightharpoonup$ denotes weak convergence. Recall that a map $T$ on $\mathbb{R}^d$ is a \emph{similarity} if there is a constant $c  \in \left(0,1 \right)$ such that  for all $x,y \in \mathbb{R}^d$ we have $\lvert T(x)-T(y) \rvert= c\lvert x-y \rvert$.  We refer to $c$ as the \emph{similarity ratio}. 

\emph{Tangent measure} usually refers to a weak limit of magnifications of a given measure at a specific point, in contrast to weak tangent measures where the point of magnification can move around.   In the seminal work of Preiss \cite{preiss} tangent measures were allowed to have unbounded support and when one restricts the magnifications to a fixed compact set, the (compactly supported)  limit measures are  often called \emph{micromeasures}, following the important work of Furstenberg \cite{furstenberg} and Hochman \cite{hochman}.  

\begin{thm}\label{weaktangents}
Suppose $\hat{\mu}$ is a weak tangent measure to a locally finite Borel measure  $\mu$ on $\mathbb{R}^d$.  Then $\overline{\dim}_{\textup{reg}} \hat{\mu} \le \r $.
\end{thm}

We will prove this result in Section \ref{weaktangentsproof}.  An immediate corollary  is that weak tangent measures of doubling measures are doubling.  This observation generalises the folklore result that `tangent measures of doubling measures are doubling'.  

\begin{cor}
All weak tangent measures of  doubling Borel measures are doubling.
\end{cor}

In our definition of weak tangent measures we follow the strategy of Preiss, allowing the tangent to have unbounded support. However the weak tangent (sets) for Assouad dimension follow the conventions of Furstenberg and Hochman, restricting to a compact set. It turns out that this difference in approach is necessary as the dimension of a set cannot increase under restriction but the dimension of a measure can. The general problem of determining when a doubling measure restricts to a doubling measure on a compact set of positive mass is a subtle problem, see \cite{ojala}.  We provide a simple example of a doubling measure on the plane which upon restriction to the unit ball becomes non-doubling.  In particular, this shows that if, in the definition of weak tangent measures, one restricts the measures $p_k \mu \circ T^{-1}_k$ to the unit ball, the analogue of Theorem \ref{weaktangents} generally fails. This is because the restricted measure would become a weak tangent measure to the original measure.

We begin with the square $X_0 = [-3/2, 3/2]^2$ and cut out a sequence of open   `almost' semicircles as follows.   Let $\left\{r_i \right\}_{i\in \mathbb{N}}$ be a sequence of radii which decay exponentially to 0 and $\left\{x_i\right\}_{i\in \mathbb{N}}$ be a sequence of centres moving clockwise on  $S^1$ which converge polynomially to a limit on the opposite side of $S^1$ from $x_1$ (without making any full rotations). Consider the balls $B_i = B(x_i, r_i)$ and remove from $X_0$ the points in the interior of  $B_i \cap B(0,1)$ which are at distance strictly greater than $r_i^2$ from $S^1$, see the portion shown in grey on the left in Figure \ref{example}.  After all of these portions have been removed, label the remaining set as $X$.  The complement of $X$ is shown in grey on the right of Figure \ref{example}. 

\begin{figure}[h]
\centering
\includegraphics[width=0.9\textwidth]{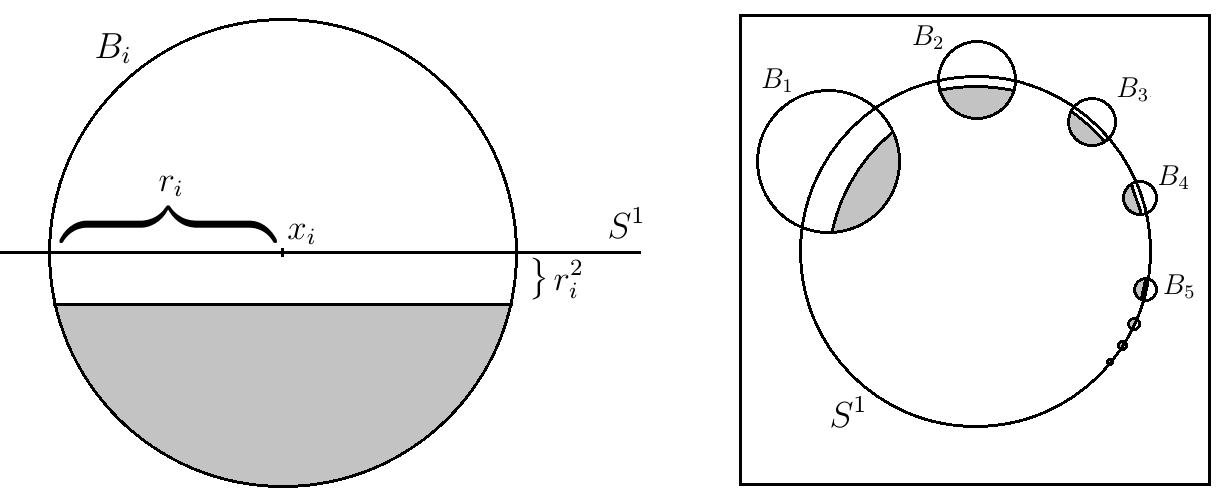}
\caption{Construction of a doubling measure which is not doubling after restricting to the unit ball}
\label{example}
\end{figure}

We observe  that 2-dimensional Lebesgue measure $\mu$ on $X$ is doubling, but that the restriction of $\mu$ to $B(0,1)$ is not doubling.  It is easy to see that there exists a uniform constant $c>0$ such that for $x \in X$ and $r\in (0,1)$ we have $cr^2 \leq \mu(B(x,r)) \leq \pi r^2$.  It follows that $\mu$ is doubling (with upper regularity dimension 2).  We now consider the restricted measure $\nu = \mu \vert_{B(0,1)}$ and compare the masses of $B(x_i,r_i)$ and $B(x_i,2r_i)$. For large enough $i$ we have $\nu (B(x_i,2r_i)) \ge (\pi (2r_i)^2-\pi r_i^2)/3 = \pi r_i^2$  and $\nu (B(x_i,r_i)) \le 4 r_i^3$. Therefore
\[
\frac{\nu(B(x_i,2r_i))}{\nu(B(x_i,r_i))} \ge \frac{\pi r_i^2}{4r_i^3} =  \frac{\pi}{4r_i} \to \infty
\]
and so  $\nu$ is not doubling.

\subsection{Self-similar measures}\label{self-similarresult}

In this section we compute the upper regularity dimension of self-similar measures satisfying the strong separation condition.  We emphasise that this separation assumption is natural because self-similar measures not satisfying the strong separation condition are typically not doubling and so have upper regularity dimension equal to $+\infty$.

Let  $\mathcal{I}$ be a finite index set and $\left\{S_i \right\}_{i \in \mathcal{I}}$ be a finite collection of  contraction maps on a compact subset of $\mathbb{R}^d$.  Such a collection is known as an iterated function system (IFS).  Also let $\{p_i\}_{i \in \mathcal{I}}$ be a collection of probabilities associated with the maps $\{S_i\}_{i \in \mathcal{I}}$, i.e. we assume that for each $i \in \mathcal{I}$ we have $p_i>0$ and $\sum_{i \in \mathcal{I}} p_i = 1$. There is a  unique non-empty compact set $F$ satisfying 
\[
F=\displaystyle\bigcup_{i\in \mathcal{I}} S_i(F)
\]
and  a unique Borel probability measure $\mu$ satisfying
\[
\mu = \sum_{i \in \mathcal{I}} p_i \mu \circ S_i^{-1}
\]
which is fully supported on $F$, see \cite[Chapter 9]{falconer} and the references therein.  When all of the contractions $S_i$ are similarities, with similarity ratio $c_i \in \left(0,1 \right)$, then $F$ is called a \emph{self-similar set} and $\mu$ is called a \emph{self-similar measure}.  We refer the reader to \cite{falconer} for a more in depth discussion of IFSs and self-similar sets and measures.

\begin{figure}[h]
\centering
\begin{subfigure}{0.3\textwidth}
  \centering
  \includegraphics[width=0.85\linewidth]{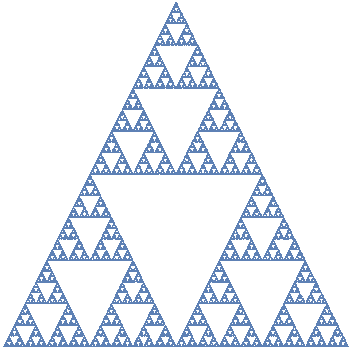}
\end{subfigure}%
\begin{subfigure}{.3\textwidth}
  \centering
  \includegraphics[width=.9\linewidth]{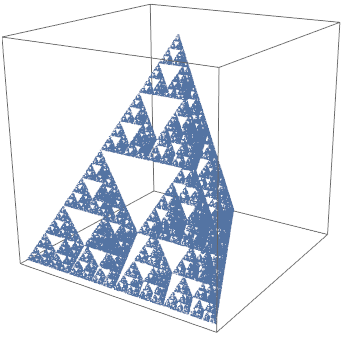}
\end{subfigure}%
\begin{subfigure}{.3\textwidth}
  \centering
  \includegraphics[width=.9\linewidth]{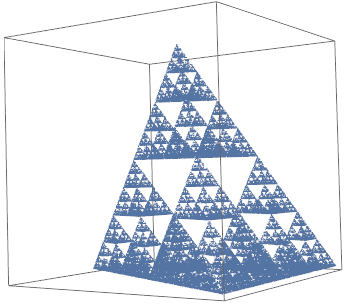}
\end{subfigure}
\caption{A self-similar Sierpi\'nski triangle and two perspectives of a self-similar Sierpi\'nski tetrahedron.}
\label{fig:test}
\end{figure}

We say that the IFS (and associated set $F$ and measure $\mu$) satisfy the \emph{strong separation condition (SSC)} if $S_i(F) \cap S_j(F) = \emptyset$ for all distinct $i,j \in \mathcal{I}$.  This is a natural assumption in the context of the upper regularity dimension.  For example, if the defining IFS consists of the maps $x \mapsto x/2$ and $x \mapsto x/2+1/2$, then the SSC is not satisfied and one easily verifies that $\mu$ is doubling if and only if both probabilities are equal to 1/2 and in this case $\mu$ is Ahlfors-David 1-regular (it is Lebesgue measure on the unit interval).

\begin{thm}\label{selfsimilar}
Let $\mu$ be a self-similar measure as defined above and assume $\mu$ satisfies the SSC.  Then 
\[
\r = \sup_x \overline{\dim}_{\textup{loc}}(x,\mu) = T(\mu) = \max_{i \in \mathcal{I}} \frac{\log p_i}{\log c_i}.
\]
\end{thm}

The Assouad dimension of the self-similar set which supports $\mu$ is generally strictly smaller than the upper regularity dimension of $\mu$ in this setting.  In fact the only case where the Assouad dimension and upper regularity dimensions coincide is when  $p_i=c_i^s$ where $s$ is the unique solution of $\sum_{i \in \mathcal{I}} c_i^s = 1$.   In this case $\mu$ is Ahlfors-David $s$-regular and all of the notions of dimension for $F$ and $\mu$ coincide and equal $s$.

\subsection{Self-affine measures}\label{self-affineresults}

In this section we consider an important class of self-affine measures.  Self-affine measures are defined in a similar way to the self-similar measures considered in the previous section, the only difference being that the defining contractions are assumed to be affinities rather than similarities.  In general, such measures are much more difficult to handle due to the fact that different rates of distortion can occur in different directions.  The specific class of self-affine measures we consider are those supported on Bedford-McMullen carpets, see \cite{bedford, mcmullen}, and on the higher dimensional analogues, the Bedford-McMullen sponges, see \cite{kenyonperes, sponges}.  These sets (and measures) came to prominence recently when they were used by Das and Simmons to provide a counter example to an important and long-standing conjecture in dynamical systems \cite{das-simmons}. In particular, there exists a surprising example of a sponge in $\mathbb{R}^3$ whose Hausdorff dimension cannot be approximated by the Hausdorff dimension of measures invariant under the natural associated dynamical system.  We assume the measures satisfy the \emph{very strong separation condition}, which was used by Olsen in \cite{sponges}.  Again, this is the natural condition to assume in the context of upper regularity dimension because without this assumption the self-affine measures tend not to be doubling, see \cite{doublingcarpets, fraser-howroyd}.

Let $d \geq 2$ be an integer and fix integers $1<n_1 < n_2 \cdots < n_d$.  Choose a subset $\mathcal{I}$ of $\prod_{l=1}^{d} \left\lbrace 0,\ldots, n_l-1 \right\rbrace$ and for $\textbf{i}=(i_1, \ldots, i_d)\in \mathcal{I} $  let $S_{\textbf{i}} \colon [0,1]^d \rightarrow [0,1]^d$ be defined by
\[
S_{\textbf{i}}(x_1,x_2\ldots, x_d)= \left( \frac{x_1+i_1}{n_1}, \frac{x_2+i_2}{n_2}, \ldots, \frac{x_d+i_d}{n_d} \right) .
\]
 Finally, consider the IFS $\{S_i\}_{i \in \mathcal{I}}$  acting on $[0,1]^d$ and let $\{p_i\}_{i \in \mathcal{I}}$ be an associated probability vector as before.  Let $F$ be the associated attractor of this IFS, which is a self-affine set since each of the defining contractions is an affinity, and let $\mu$ be the associated self-affine measure.  We can now state the separation condition we require, which we note  is  strictly stronger than the SSC.

\begin{figure}[h]
\centering
\begin{subfigure}{0.35\textwidth}
  \centering
  \includegraphics[width=0.8\linewidth]{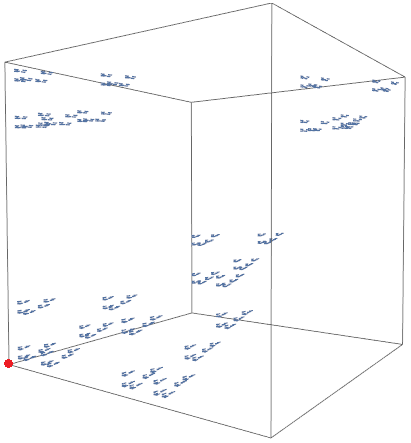}
  \label{fig:sub1}
\end{subfigure}%
\begin{subfigure}{0.35\textwidth}
  \centering
  \includegraphics[width=0.85\linewidth]{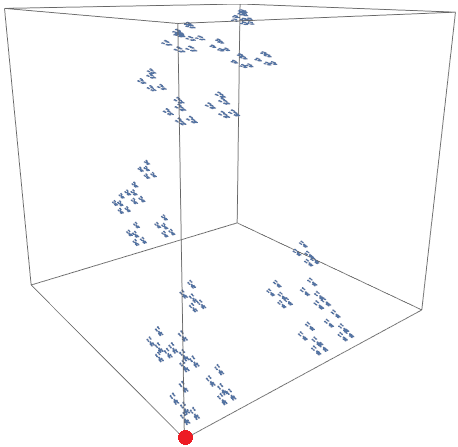}
  \label{fig:sub2}
\end{subfigure}%
\begin{subfigure}{0.35\textwidth}
  \centering
  \includegraphics[width=0.85\linewidth]{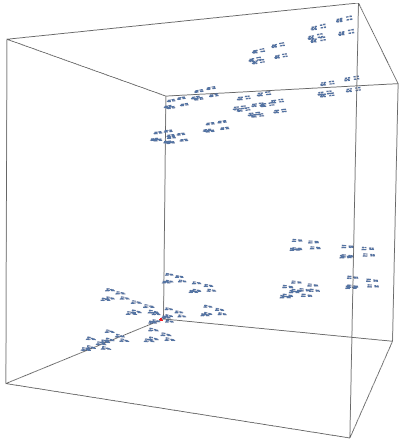}
  \label{fig:sub2}
\end{subfigure}
\caption{Three perspectives of the self-affine sponge defined by the data: $d=3$,  $n_1=3$, $n_2=4$, $n_3=5$ and  $\mathcal{I} = \{ (0,0,0),(0,2,0),(2,1,1)$, $(2,3,4)$, $(0,0,4) \}$. The origin is marked by a red dot to indicate orientation.}
\label{fig:test}
\end{figure}

\begin{defn}[VSSC, \cite{sponges}]
\emph{A self-affine sponge }$F$\emph{ (associated to an index set }$\mathcal{I}$)\emph{ satisfies the }very strong separation condition (VSSC)\emph{ if the following condition  holds. If }$l=1, \ldots, d$\emph{ and }$(i_1, \ldots, i_d)$, $(j_1 ,\ldots, j_d)\in \mathcal{I}$\emph{ satisfy }$i_1=j_1, \ldots,  i_{l-1}=j_{l-1}$\emph{ and }$i_l \neq j_l$,\emph{ then }$\lvert i_l - j_l \rvert >1$.
\end{defn} 

Before we state our result, we need to introduce some more notation.  For $l=1, \dots, d$ and $\textbf{i}=(i_1, \ldots, i_d)\in \mathcal{I} $  let 
\[
p_l(\mathbf{i})=p (i_l \vert i_1, \ldots , i_{l-1})=\frac{\displaystyle\sum_{\substack{\textbf{j}=\left( j_1, \ldots, j_d\right)\in \mathcal{I} \\ j_1=i_1, \ldots, j_{l-1}=i_{l-1}, j_l=i_l}}p_{\textbf{j}}}{\displaystyle\sum_{\substack{\textbf{j}=\left( j_1, \ldots, j_d\right)\in \mathcal{I} \\ j_1=i_1, \ldots, j_{l-1}=i_{l-1}}}p_{\textbf{j}}}
\]
if $(i_1, \ldots, i_l, i_{l+1},\ldots, i_d) \in \mathcal{I}$ for some $i_{l+1},\ldots, i_d$ and $0$ otherwise.  These numbers have a clear interpretation: $p_l(\mathbf{i})$ is the conditional probability that the $l$th digit of an element of $\mathcal{I}$  coincides with the $l$th digit of $\mathbf{i}$, given that the first $l-1$ coordinates did.  Note that when $l=1$ we are conditioning on the entire space and so the denominator of the above conditional probability is taken to be 1.

\begin{thm}\label{carpet}
Let $\mu$ be a self-affine measure on a Bedford-McMullen sponge satisfying the VSSC.  Then
\[
\r =\sum_{l=1}^d \max_{\mathbf{i}\in \mathcal{I}}\frac{-\log p_l(\mathbf{i})}{\log n_l}.
\]
\end{thm}

Formulae for $\sup_x \overline{\dim}_{\text{loc}}(x,\mu)$ and $T(\mu)$ for the self-affine measures $\mu$ we consider in this section can be found in \cite{sponges}, where the notation $\overline{a}$ and $\overline{A}$ was used, respectively.    Also, the Assouad dimension of these sponges was first computed by Mackay \cite{mackay} in the case $d=2$ and in \cite{fraser-howroyd} for general $d$.

We will now discuss a family of examples designed to demonstrate that all of the notions of dimension we discuss here can be distinct for self-affine measures.  In particular, the upper regularity dimension can be strictly greater than the Assouad dimension, supremum of the upper local dimensions and the `top of the spectrum', $T(\mu)$.  This behaviour was \emph{not} seen in the self-similar case.

Let $d=2$, $n_1=3$, $n_2=4$,  $\mathcal{I}=\left\{(0,2),(2,1),(2,3)\right\}$ and $p_{(0,2)}=\epsilon,\, p_{(2,1)}=1-3\epsilon/2$ and $p_{(2,3)}=\epsilon/2$ where we allow $\epsilon$ to vary in the interval $(0,1/2]$.  We write $F$ for the self-affine carpet and $\mu$ for the self-affine measure associated with this data.  Observe that the VSSC is satisfied and so our results apply.  Theorem \ref{carpet} yields
\[
\r = \frac{-\log \epsilon}{\log 3}+ \frac{-\log \frac{\epsilon/2}{1-\epsilon} }{\log 4},
\]
Mackay's result \cite{mackay} gives
\[
\dim_{\text{A}} F = \frac{\log 2}{ \log 3} + \frac{ \log 2}{ \log 4}, 
\]
and, using results from \cite{sponges},
\[
\sup_{x\in F} \overline{\dim}_{\text{loc}}(x,\mu)=\max\left\{ \frac{-\log \epsilon}{\log 3} , \frac{-\log (1-\epsilon)}{\log 3}+\frac{-\log \frac{\epsilon/2}{1-\epsilon} }{\log 4}\right\},
\]
(which has a phase transition at $\epsilon \approx 0.066$), and
\[
T(\mu)= \frac{-\log \epsilon }{\log 3} +\frac{\log 2}{\log 4} .
\]
For $\epsilon \in (0,1/2)$  these quantities are all distinct and for $\epsilon=1/2$ the measure is the `coordinate uniform measure' from \cite{fraser-howroyd} which has upper regularity dimension precisely equal to the Assouad dimension.

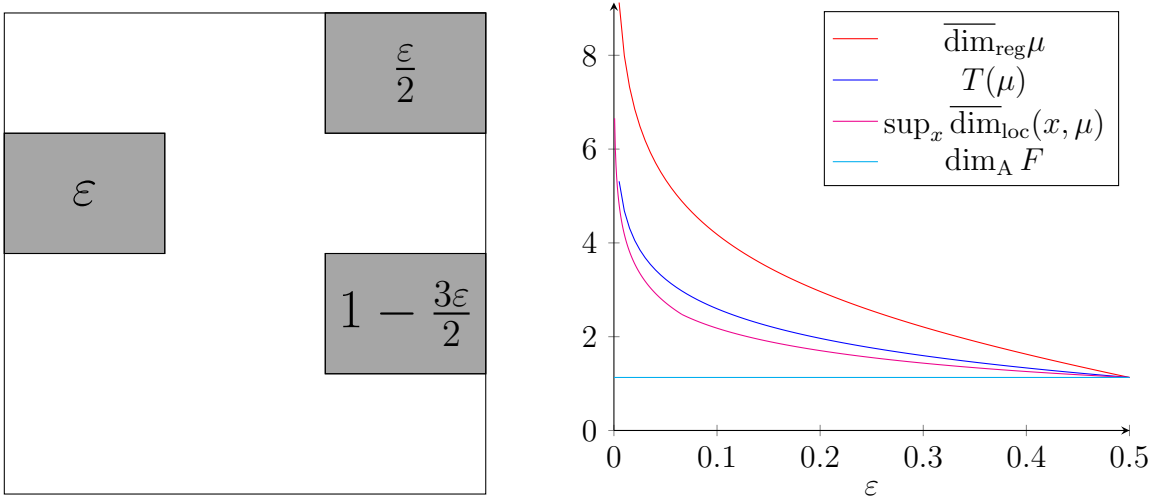
\begin{figure}[H]
\centering
\begin{minipage}{.5\textwidth}
  \centering
\begin{tikzpicture}[scale=0.8]
\coordinate (O) at (0,0);
\coordinate (A) at (\Width,0);
\coordinate (B) at (\Width,\Height);
\coordinate (C) at (0,\Height);
\coordinate (C1) at (0,\Height*2/4);
\coordinate (C2) at (0,\Height*3/4);
\coordinate (C3) at (\Width/3,\Height*2/4);
\coordinate (C4) at (\Width/3,\Height*3/4);
\coordinate (A1) at (\Width*2/3,\Height/4);
\coordinate (A2) at (\Width*2/3,\Height*2/4);
\coordinate (A3) at (\Width,\Height/4);
\coordinate (A4) at (\Width,\Height*2/4);
\coordinate (D1) at (\Width*2/3,\Height*3/4);
\coordinate (D2) at (\Width*2/3,\Height);
\coordinate (D3) at (\Width,\Height*3/4);
\coordinate (D4) at (\Width,\Height);
\draw[black] (O) -- (A) -- (B) -- (C) -- cycle;
\draw[black,fill=gray!70] (C1) -- (C2) -- (C4) -- (C3) -- cycle;
\draw (C1) rectangle (C4) node[pos=.5] {\LARGE$\varepsilon$};
\draw[black,fill=gray!70] (A1) -- (A2) -- (A4) -- (A3) -- cycle;
\draw (A1) rectangle (A4) node[pos=.5] {\LARGE$ 1 -  \frac{3\varepsilon}{2}$};
\draw[black,fill=gray!70] (D1) -- (D2) -- (D4) -- (D3) -- cycle;
\draw (D1) rectangle (D4) node[pos=.5] {\LARGE $\frac{\varepsilon}{2}$};
\end{tikzpicture}

\end{minipage}%
\begin{minipage}{.5\textwidth}
  \centering

\begin{tikzpicture}[yscale=1]
\begin{axis}[
    axis lines = left,
    xlabel = $\varepsilon$,
]
\addplot [
    domain=0:0.5, 
    samples=100, 
    color=red,
]
{(-ln(x))/(ln(3)) + (-ln((x/2)/(1-x)))/(ln(4)) };
\addlegendentry{$\r$}

\addplot [
    domain=0:0.5, 
    samples=100, 
    color=blue,
    ]
    {(-ln(x))/(ln(3)) + ln(2)/ln(4)};
\addlegendentry{$T(\mu)$}

\addplot [
    domain=0.066:0.5, 
    samples=100, 
    color=magenta,
]
{- ln((x/2)/(1-x))/ln(4)- ln(1-x)/ln(3)  };
\addlegendentry{$\sup_x \l$}

\addplot [
    domain=0:0.5, 
    samples=100, 
    color=cyan,
]
{(ln(2))/(ln(3)) + (ln(2))/(ln(4)) };
\addlegendentry{$\dim_\text{A} F$}

\addplot [
    domain=0:0.5, 
    samples=100, 
    color=black,
]
{0 };

\addplot [
    domain=0:0.066, 
    samples=100, 
    color=magenta,
]
{  ln(1/x)/ln(3)   };

\end{axis}
\end{tikzpicture}

\end{minipage}

  \caption{Left: the affine maps and associated probabilities.  Right: a plot showing the four different `dimensions' as $\epsilon$ varies.}
  \label{fig:badcarpet}
\end{figure}

\subsection{Measures on sequences}

When one first meets the Assouad dimension, the first interesting example is often that the set $\left\{1/n \right\}_{n\in \mathbb{N}}$ has Assouad dimension 1, which is strictly larger than the upper and lower box dimensions which are both 1/2.  Since this example is so prevalent, we decided to investigate the upper regularity dimensions of natural families of measures supported on such sets. For simplicity we restrict our examples to countable subsets of $[0,1]$ with one accumulation point at 0 and measures equal to the sum of decaying point masses on the elements of the set. The interplay between the rate of convergence of the points in the set and the rate of decay of the point masses will turn out to be paramount to understanding the dimension of the measure, and to emphasise this we provide exact results for some simple cases where the rates of convergence are either polynomial or exponential.  However, it could be interesting in the future to study  sequences with other decay rates, for example  stretched exponential decay $a^{n^b}$ for $a,b \in (0,1)$, $n \in \mathbb{N}$.

More concretely, consider the set $\left\{ x_n \ : \ n \in \mathbb{N}\right\}$ where  $x_n  \searrow 0$ and the sequence of weights $\left\{ p(n) \ : \ n \in \mathbb{N} \right\}$ where $p(n) \searrow 0$, and $\sum_{n=1}^\infty p(n) < \infty$.  The measure we are interested in is
\[
\mu = \frac{1}{\sum_{n=1}^\infty p(n) } \sum_{n=1}^\infty p(n)\delta_{x_n} 
\]
where $\delta_{x_n} $ is a point mass at $x_n$.

\begin{thm}\label{sequences}
Let $\mu$ be as above.
\begin{enumerate}
\item Polynomial-polynomial: Let $\lambda > 0$ and $\omega > 1$ and suppose $x_n = n^{-\lambda}$ and $p(n)=n^{-\omega}$.  Then
\[
\r \ = \  \max \left\{1, \, \frac{\omega-1}{\lambda}\right\} \  =\  \max \left\{\a \textup{supp} (\mu), \, \sup_x \overline{\dim}_{\textup{loc}}(x,\mu)\right\}.
\]
\item Exponential-exponential: Let $\lambda, \omega \in (0,1)$ and suppose $x_n= \lambda^{n}$ and $p(n)=\omega^{n}$.  Then
\[
\r \ = \  \frac{\log \omega}{\log \lambda}  \ = \ \sup_x \overline{\dim}_{\textup{loc}}(x,\mu) \ >  \     0 \ = \ \a \textup{supp} (\mu) .
\]
\item Mixed rates: If
\begin{enumerate}
\item[(i)] $x_n = n^{-\lambda}$ $(\lambda >0)$ and $p(n)=\omega^{n}$ $(0< \omega < 1)$; or
\item[(ii)]  $x_n =  \lambda^{n}$ $(0< \lambda < 1)$ and $p(n)=n^{-\omega}$ $(\omega >1)$,
\end{enumerate}
then $\mu$ is not doubling, and so  $\r = \infty$.
\end{enumerate}
\end{thm}

The above theorem can be summarised by the following table, for suitable values of $\omega$ and $\lambda$:

\begin{table}[h]
\centering
\label{sequencetable}
\begin{tabular}{c|cc}
$p(n) \setminus x_n$ & $n^{-\lambda}$             & $\lambda^n$                        \\ \hline
$n^{-\omega}$       & $\max \left\{1,\frac{\omega - 1}{\lambda}\right\}$ & $\infty$\\
$\omega^n$          & $\infty$                         & $\frac{\log \omega}{\log \lambda}$
\end{tabular}
\end{table}

\section{Proofs} \label{proof}

We prove Theorem \ref{relationships} in Section \ref{spectrumproof} followed by Theorem \ref{weaktangents} on weak tangents in Section \ref{weaktangentsproof}. Section \ref{self-similar} will concern self-similar measures and will include a proof of Theorem \ref{selfsimilar}. Self-affine measures and the proof of Theorem \ref{carpet} will be dealt with in Section \ref{self-affine} along with some additional notation needed to study Bedford-McMullen sponges. Finally, Theorem \ref{sequences} will be proved in Section \ref{sequenceproof}. Any notation introduced in a subsection should only be used in that proof, but any notation used in the first two sections is assumed throughout.

\subsection{Proof of equivalence of doubling and finite upper regularity dimension} \label{doublingproof}

As we mentioned in the introduction, a measure is doubling if and only if it has finite upper regularity dimension. Here we state a quantifiable version of this fundamental fact and include our own short proof for completeness.  Given $\theta \in (0,1)$ let
\[
C(\theta) = \sup \left\{ \frac{\mu(B(x,R))}{\mu(B(x,\theta R))} \ : \ x \in \text{supp}(\mu), R>0 \right\} \geq 1
\]
and note that $C(1/2)$ is the `doubling constant' from the definition of doubling.  Also recall that if $C(\theta) < \infty$ for some $\theta \in (0,1)$, then  $C(\theta) < \infty$ for all $\theta \in (0,1)$.

\begin{prop}\label{doubling2}
A locally finite Borel measure is doubling if and only if $\r<\infty$.  Moreover, we have
\[
\r \le \inf_{\theta \in (0,1)} \frac{ \log C(\theta)}{ - \log \theta}.
\]
\end{prop}

\begin{proof}
If $\r <\infty$ then it follows immediately from the definition that $\mu$ is doubling.  For the reverse implication, assume that $\mu$ is doubling and let $\theta \in (0,1)$.  Therefore, for all $x \in \text{supp}(\mu)$ and $R>0$
\[
\frac{\mu(B(x,R))}{\mu(B(x,\theta R))} \le C(\theta).
\]
Fix  $0< r < R$ and  define $k$ to be the unique integer such that $R\theta^k > r \ge R\theta^{k+1}$. Then by telescoping we obtain
\[
\frac{\mu(B(x,R))}{\mu(B(x,r))} = \frac{\mu(B(x,R))}{\mu(B(x,\theta R))} \times \frac{\mu(B(x,\theta R))}{\mu(B(x,\theta^2 R))} \times\cdots \times \frac{\mu(B(x,\theta^k R))}{\mu(B(x,\theta^{k+1}R))} \times \frac{\mu(B(x,\theta^{k+1}R))}{\mu(B(x,r))}.
\]
Thus 
\[
\frac{\mu(B(x,R))}{\mu(B(x,r))} \le C(\theta)^{k+1}  \frac{\mu(B(x,\theta^{k+1}R))}{\mu(B(x,r))} \le C(\theta)^{\log(r/R) / \log \theta+1} = C(\theta) \left(\frac{R}{r} \right)^{-\log C(\theta) / \log \theta}
\]
and hence $\r \leq -\log C(\theta) / \log \theta < \infty$, as required.
\end{proof}

\subsection{Proof of Theorem \ref{relationships}: general relationships}\label{spectrumproof}

Let $\mu$ be a Borel probability measure supported on a compact set $X \subseteq \mathbb{R}^d$. Let $0<s< T(\mu)$,  $\r < t < \infty$ and $q<0$. By definition there exists a  constant  $C \geq 1$ such that for all $x\in X$ and for all $0<r< 1$
\[
\frac{\mu(B(x,1))}{\mu(B(x,r))} \le C \left(\frac{1}{r} \right)^t.
\]
In particular, this guarantees 
\[
\mu(B(x,r))^q \le \frac{1}{C^q}\mu(B(x,1))^q  r^{qt}
\]
and, moreover, 
\[
M_r^q(\mu) \leq c r^{-d}r^{qt}
\]
where $c>0$ is a constant independent of  $r$ and where the $r^{-d}$ term comes from an upper bound on $r$-packings of $X \subseteq \mathbb{R}^d$.  Therefore 
\[
sq > \underline{\tau}(q)   \ge  qt-d
\]
and so $s<t-d/q$ for any $q<0$. By letting $q \rightarrow - \infty$ this yields $s \le t$, which is sufficient to prove that $T(\mu) \leq \r$.

 All that remains is to prove that $T(\mu) \geq \l$ for all $x \in X$.  As such, let $x \in X$ and $u>T(\mu)$ which implies that $\underline{\tau}(q)  \geq  uq$ for some $q<0$ which we fix.  Therefore, given $\epsilon>0$, there exists a constant $C_\epsilon>0$ such that
\[
M_r^q(\mu) \leq C_\epsilon r^{uq-\epsilon}
\]
for all $r \in (0,1)$.  Since $\{ B(x,r)\}$ is an $r$-packing of $X$, it follows that
\[
\mu(B(x,r))^q \le M_r^q(\mu) \leq C_\epsilon r^{uq-\epsilon}
\]
and therefore
\[
\mu(B(x,r)) \ge  C_\epsilon^{1/q} r^{u-\epsilon/q}
\]
for all $r\in (0,1)$, which proves that $\l \le u-\epsilon/q$ and since $\epsilon>0$ was arbitrary, this completes the proof.

\subsection{Weak tangent measures} \label{weaktangentsproof}

Let $\mu$ be a locally finite Borel measure on $\mathbb{R}^d$, $\left\{T_k\right\}_{k\in \mathbb{N}}$ a sequence of similarities on $\mathbb{R}^d$ with associated contraction ratios $\{c_k\}_{k\in \mathbb{N}}$, $\left\{p_k\right\}_{k\in \mathbb{N}}$ a sequence of positive renormalising numbers, and $\hat{\mu}$ be a corresponding weak tangent measure of $\mu$, that is a Borel measure on $\mathbb{R}^d$ such that 
\[
p_k \mu \circ T^{-1}_k \rightharpoonup \hat{\mu}
\]
where $\rightharpoonup$ means weak convergence. The Portmanteau Theorem (see \cite[Theorem 1.24]{mattila}) says that this is equivalent to 
\[
\lim_{k\rightarrow \infty} p_k \mu \circ T^{-1}_k (A) =\hat{\mu}(A)
\]
for all $\hat{\mu}$-continuity sets $A$.  Recall that  $A\subseteq \mathbb{R}^d$ is a $\hat{\mu}$-continuity set when $\hat{\mu}(\partial A) = 0$ with $\partial A$ being the boundary of $A$. It is a simple exercise to show that for a fixed $x \in \mathbb{R}^d$, all but at most countably many balls $B(x,r)$ are $\hat{\mu}$-continuity sets of $\mathbb{R}^d$. We now provide a technical lemma which reduces our calculation of the upper regularity dimension of $\hat \mu$ to the study of balls which are $\hat{\mu}$-continuity sets.

\begin{lma}
Let $\nu$ be a locally finite Borel measure  on $ \mathbb{R}^d$. Suppose there exist constants $C$ and $s$ such that for all $x\in \textup{supp}(\nu)$ and $0<r<R$ such that $B(x,R)$ and $B(x,r)$ are $\nu$-continuity sets, we have
\[
\frac{\nu(B(x,R))}{\nu(B(x,r))} \leq C\left(\frac{R}{r}\right)^{s}.
\]
Then $\overline{\dim}_{\textup{reg}}  \nu \le s.$
\end{lma}

\begin{proof}
Assume 
\[
\frac{\nu(B(x,R))}{\nu(B(x,r))} \leq C\left(\frac{R}{r}\right)^{s}
\]
holds for all $\nu$-continuity balls. Fix $x  \in \textup{supp}(\nu)$ and let $0<r<R$ be arbitrary. Since there are at most countably many problematic radii, there must exist constants $\theta_1(x,r),\theta_2(x,R)\in[1,2]$ such that $B(x,\theta_1(x,r)^{-1}r)$ and $B(x,\theta_2(x,R) R)$ are $\nu$-continuity balls. Thus
\[
\frac{\nu(B(x,R))}{\nu(B(x,r))} \leq \frac{\nu(B(x,\theta_2(x,R)R))}{\nu(B(x,\theta_1(x,r)^{-1}r))} \le   C\left(\frac{\theta_2(x,R)R}{\theta_1(x,r)^{-1}r}\right)^{s} \le 4^s C \left(\frac{R}{r}\right)^{s}
\]
and it follows that $\overline{\dim}_{\textup{reg}}  \nu \le s.$
\end{proof}

We now return to proving Theorem \ref{weaktangents}. Let  $x \in \textup{supp}( \hat \mu)$ and $\rho>0$ be such that $B(x,\rho)$ is a $\hat \mu$-continuity set.   Therefore 
\[
\lim_{k\rightarrow \infty} p_k \mu \circ T^{-1}_k (B(x,\rho)) =\hat{\mu}(B(x,\rho)).
\]
Therefore, for sufficiently large $k$, 
\begin{equation} \label{estimatee}
\frac{1}{2}p_k \mu \circ T^{-1}_k (B(x,\rho))  \le \hat{\mu}(B(x,\rho))\le 2p_k \mu \circ T^{-1}_k (B(x,\rho)).
\end{equation}
  Let $\varepsilon > 0 $ and  $0<r<R$ be such that both $B(x,R)$ and $B(x,r)$ are $\hat \mu$-continuity sets  and  choose $k $ large enough so that  (\ref{estimatee}) holds for $\rho=r$ and $\rho=R$.  In particular,  
\[
\frac{\hat{\mu}(B(x,R))}{\hat{\mu}(B(x,r))} \le 4 \frac{p_k\mu \circ T^{-1}_k (B(x,R))}{p_k \mu \circ T^{-1}_k (B(x,r))}=4\frac{\mu (T^{-1}_k (B(x,R)))}{\mu( T^{-1}_k (B(x,r)))}.
\]
Note that $T_k$ is a similarity of contraction ratio $c_k >0$ and so $T_k^{-1}(B(x,R))= B(T_k^{-1}(x),c_k^{-1}R)$. Thus
\begin{align*}
\frac{\hat{\mu}(B(x,R))}{\hat{\mu}(B(x,r))} &\le 4\frac{\mu (B(T_k^{-1}(x),c_k^{-1}R))}{\mu (B(T_k^{-1}(x),c_k^{-1}r))}.
\end{align*}
Here we wish to apply the definition of the upper regularity dimension of $\mu$, but we cannot do this directly since  $T_k^{-1}(x)$ does not have to be in $\text{supp}(\mu)$. However, we can assume $k$ is large enough (depending on $r$) so that there exists $x'$ in the support of $\mu \circ T_k^{-1}$ which is at distance at most $r/2$ from $x$. Therefore
\[
B(T_k^{-1}(x'),c_k^{-1}r/2)\subset B(T_k^{-1}(x),c_k^{-1}r) \qquad \text{ and } \qquad   B(T_k^{-1}(x),c_k^{-1}R) \subset B(T_k^{-1}(x'), 2c_k^{-1}R)
\]
and $T_k^{-1}(x')$ is in the support of $\mu$.  Therefore
\begin{align*}
\frac{\hat{\mu}(B(x,R))}{\hat{\mu}(B(x,r))} &\le 4\frac{\mu (B(T_k^{-1}(x),c_k^{-1}R))}{\mu (B(T_k^{-1}(x),c_k^{-1}r))} \le 4\frac{\mu (B(T_k^{-1}(x'),2c_k^{-1}R))}{\mu (B(T_k^{-1}(x'),c_k^{-1}r/2))} \\
&\le 4C \left(\frac{4c_k^{-1}R}{c_k^{-1}r}\right)^{\r+ \varepsilon} =  4^{1+\r+ \varepsilon}C \left(\frac{R}{r}\right)^{\r+ \varepsilon}
\end{align*}
where $C= C(\varepsilon)>0$ is a uniform constant independent of $x$, $r$ and $R$ which comes from the definition of the upper regularity dimension of $\mu$. Letting $\epsilon \to 0$ proves  $\overline{\dim}_{\textup{reg}} \hat{\mu} \le \r$ as desired.

\subsection{Proof of Theorem \ref{selfsimilar}: self-similar measures} \label{self-similar}

There is a natural correspondence between the geometric fractal $F$ and the symbolic space $\mathcal{I}^{\mathbb{N}}$ (the set of all infinite words over $\mathcal{I}$) via the coding map $\pi \colon \mathcal{I}^{\mathbb{N}} \rightarrow \mathbb{R}^d$ defined by 
\[
\{\pi(i_0,i_1,\ldots) \} =  \bigcap_{n=0}^\infty S_{i_0,\ldots, i_n}(F)
\]
where $S_{i_0,\ldots, i_n} = S_{i_0} \circ \cdots \circ S_{ i_n}$.  For $ i_0,i_1,\ldots, i_{n-1} \in \mathcal{I}$ we define a cylinder $[i_0,i_1,\ldots, i_{n-1}] \subseteq \mathcal{I}^\mathbb{N}$, to be the set of all words in $\mathcal{I}^{\mathbb{N}}$ whose first $n$ letters are $i_0,i_1,\ldots, i_{n-1}$. The collection of all level $n$ cylinders corresponds to the  $n$th level  pre-fractal of the attractor.  It is well-known that $F= \pi(\mathcal{I}^\mathbb{N})$ and the SSC also guarantees that $\pi$ is a bijection so we may interchange the symbolic and geometric spaces.  Also note that  $c_{i_0}\cdots c_{i_{n}}$ is the contraction ratio of the similarity $S_{i_0,\ldots, i_n} $,  $\pi([i_0,i_1,\ldots, i_{n}] ) = S_{i_0,\ldots, i_n} (F)$ and the $\mu$ measure of $S_{i_0,\ldots, i_n} (F)$  is $p_{i_0}\cdots p_{i_{n}}$.  By rescaling if necessary, we may assume without loss of generality that diam$(F)=1$.

We define $\delta$ to be the minimal distance between distinct  sets $S_i(F)$ and $S_j(F)$, that is 
\[
\delta= \min_{i\neq j} \inf_{\substack{x \in S_i(F) \\ y \in S_j(F)}} \lvert x-y \rvert.
\]
The SSC guarantees that $\delta > 0$. Thus the minimal distance between $S_{i_0,\ldots,i_{n-1},i}(F)$ and $S_{i_0,\ldots,i_{n-1},j}(F)$ is at least $c_{i_0}\cdots c_{i_{n-1}} \delta$ for any   $ i_0,i_1,\ldots, i_{n-1} \in \mathcal{I}$ and $i\neq j$.

For $x\in F$ with $\pi(i_0, i_1, \dots)=x$ for a unique $(i_0,i_1,\dots) \in \mathcal{I}^\mathbb{N}$ and small $r>0$, we define the integer $n(x,r)$ to be the largest integer such that $r \le c_{i_0}c_{i_1} \cdots c_{i_{n(x,r)}}$ and so
\[
c_{i_0}c_{i_1} \cdots c_{i_{n(x,r)}+1} < r \le c_{i_0}c_{i_1} \cdots c_{i_{n(x,r)}}.
\]
We also let $m(x,r)$ be the smallest non-negative  integer such that 
$$
\pi([i_0,\ldots, i_{m(x,r)}])\subset B(x,r/2),
$$
and so, in particular,  $p_{i_0}\cdots p_{i_{m(x,r)}} \le \mu (B(x,r))$.  Note that for any $x \in F$, and small $ r>0$,  $c_{i_0} \cdots c_{i_{m(x,r)}} \leq r$. Also  $\pi([i_0,\ldots, i_{m(x,r)-1}]) \not\subset B(x,r/2)$ and so $c_{i_0} \cdots c_{i_{m(x,r)-1}} \ge r/2$. Using the SSC, we have that for all $x=\pi(i_0,\ldots)$ and  $r>0$, we have $\mu(B(x,\frac{\delta r}{2}))\le p_{i_0}\cdots p_{i_{n(x,r)}}$ where $\delta$ is the constant determined by the SSC. This is true since $B(x, \frac{\delta r}{2}) \cap F  \subseteq \pi([i_0,\ldots, i_{n(x,r)}])$.

Let $x = \pi(\omega)$ where $\omega \in \mathcal{I}^{\mathbb{N}}$ is an infinite string of the symbol  $i$ which maximises $\log p_i/\log c_i$. It follows that $n(x,r) > \log r / \log {c_i} - 1$ and therefore
\[
\mu(B(x,\delta r/2)) \le p_{i}^{n(x,r)} \leq p_i^{-1} r^{\log p_i/\log c_i}
\]
and it follows that $\overline{\dim}_{\text{loc}}(x,\mu) \geq \log p_i/\log c_i = \colon t$.  Moreover, Theorem \ref{relationships} yields  $ \r  \geq \sup_{x\in F} \overline{\dim}_{\text{loc}}(x,\mu) \ge  t$.  We will now demonstrate the reverse inequality.

As $F$ satisfies the SSC and $\mu$ is a Bernoulli measure, it is doubling (see \cite{olsenformalism}, for example). Thus there exists a constant $C \geq 1$ depending only on $\delta/2 < 1$ such that $\frac{\mu(B(x, R))}{\mu(B(x,\delta R/2))}\le C$ for any $x\in F$ and for any $R>0$.  Let $x \in F$ and $0<r<R$  and assume without meaningful loss of generality that $n(x,R) < m(x,r)$. If this were not true, then $R/r$ is bounded above by a uniform constant -- a situation we can safely ignore.  Hence
\begin{align*}
\frac{\mu(B(x,R))}{\mu (B(x,r))}& \le \ C \frac{\mu(B(x,\delta R/2))}{\mu (B(x,r))}  \\
& \le \ C \frac{p_{i_0}\cdots p_{i_{n(x,R)}}}{p_{i_0}\cdots p_{i_{m(x,r)}}} \\
& =\  C  \frac{1}{p_{i_{n(x,R)+1}}} \cdots \frac{1}{p_{i_{m(x,r)}}} \\
& = \ C  \left(\frac{1}{c_{i_{n(x,R)+1}}}\right)^{\log p_{i_{n(x,R)+1}}/\log c_{i_{n(x,R)+1}}} \cdots \left(\frac{1}{c_{i_{m(x,r)}}}\right)^{\log p_{i_{m(x,r)}}/\log c_{i_{m(x,r)}}} \\
& \le\  C  \left( \frac{1}{c_{i_{n(x,R)+1}}}\right)^t \cdots \left( \frac{1}{c_{i_{m(x,r)}}}\right)^t \\
& = \ C  \left(\frac{1}{c_{i_{n(x,R)+1}} c_{i_{m(x,r)}}}\right)^t \left( \frac{c_{i_0}c_{i_1} \cdots c_{i_{n(x,R)+1}}}{c_{i_0}c_{i_1} \cdots c_{i_{m(x,r)-1}}}\right)^t \\
& \le \  C   \left(\frac{2}{c_{\min}^2}\right)^t \left( \frac{R}{r}\right)^t
\end{align*}
where $c_{\min} = \min_{i\in \mathcal{I}}c_i$, is just a constant.  The desired upper bound, and Theorem \ref{selfsimilar}, follows.

\subsection{Proof of Theorem \ref{carpet}: self-affine measures} \label{self-affine}

Similar to the previous section, we use the natural correspondence between the self-affine set $F$ and the symbolic space $\mathcal{I}^{\mathbb{N}}$ (the set of all infinite words over $\mathcal{I}$) via the coding map $\pi \colon \mathcal{I}^{\mathbb{N}} \rightarrow \mathbb{R}^d$ defined by 
\[
\{\pi(\textbf{i}_1, \textbf{i}_2, \ldots)\}=  \bigcap_{n=1}^\infty S_{\textbf{i}_1,\ldots, \textbf{i}_n}(F).
\]
where $ S_{\textbf{i}_1,\ldots, \textbf{i}_n} =  S_{\textbf{i}_1} \circ \cdots \circ S_{\textbf{i}_n}$.  Recall that elements of $\mathcal{I}$ have $d$ coordinates so we write $(\mathbf{i}_1,\ldots) = ((i_{1,1},\ldots, i_{1,d}),\ldots) \in \mathcal{I}^{\mathbb{N}} $.

Since the cylinders scale by different amounts in different directions, they do not directly approximate a ball in measure.  For this reason, we introduce \emph{approximate cubes}. For $r\in (0,1]$, choose the unique integers $k_1(r),\ldots,k_d(r)$, greater than or equal to 0, satisfying
\[
\frac{1}{n_l^{k_l(r)+1}}< r \leq \frac{1}{n_l^{k_l(r)}}
\]
for $l=1,\ldots,d$. In particular, 
\[
 \frac{-\log r}{\log n_l}-1 < k_l(r) \leq \frac{-\log r}{\log n_l}.
\]
Then the approximate cube $Q(\omega, r)$ of (approximate) side length $r$ determined by $\omega =\left( \textbf{i}_1, \textbf{i}_2 , \ldots \right) =\left( (i_{1,1}, \dots, i_{1,d}), (i_{2,1}, \dots, i_{2,d}) , \ldots \right)    \in \mathcal{I}^{\mathbb{N}}$ is defined by
\[
Q(\omega, r)=\left\{ \omega'=\left( \textbf{j}_1, \textbf{j}_2 , \ldots \right)\in \mathcal{I}^{\mathbb{N}} : \forall \, \,  l=1, \ldots, d \text{ and } \forall\, \, t= 1, \ldots, k_l(r) \text{ we have } j_{t,l}=i_{t,l} \right\}.
\]
The geometric analogue is $\pi\left(Q(\omega, r)\right)$, which is contained in
\[ 
\prod_{l=1}^d \left[\frac{i_{1,l}}{n_l}+\cdots+\frac{i_{k_l(r),l}}{n_l^{k_l(r)}} \, , \, \frac{i_{1,l}}{n_l}+\cdots+\frac{i_{k_l(r),l}}{n_l^{k_l(r)}}+\frac{1}{n_l^{k_l(r)}} \right];
\]
a hypercuboid in $\mathbb{R}^d$ aligned with the coordinate axes with side lengths $n_l^{-k_l(r)}$, which are all comparable to $r$ since $ r \leq n_l^{-k_l(r)} < n_l r$.  Thus the measure of a ball can be closely approximated by the measure of an appropriate approximate cube.  This is made precise by the following   useful proposition due to Olsen \cite[Proposition 6.2.1]{sponges}.
\begin{prop}[\cite{sponges}] \label{ballscubes}
Let $\omega \in \mathcal{I}^{\mathbb{N}}$ and $k \in \mathbb{N}$.
\begin{enumerate} 
\item If the VSSC is satisfied, then $B\left( \pi(\omega), 2^{-1}n_1^{-k}\right)\cap F \subseteq \pi \left(Q\left( \omega, n_1^{-k} \right) \right).$ 
\item  $\pi \left(Q\left( \omega, 1/n_1^k \right) \right) \subseteq B\left( \pi(\omega), (n_1+\dots+n_d)n_1^{-k}\right).$ 
\end{enumerate}
\end{prop}

This proposition means, since we assume the VSSC, that a ball of a particular radius contains, and is contained in, an approximate cube of a comparable radius.  Therefore we may replace balls with approximate cubes in the definition of upper regularity dimension, which makes the calculations much easier.    For  a more in depth explanation of this simplification, see \cite{fraser-howroyd}.

Recalling the conditional probabilities $p(i_{l}\vert i_{1},\ldots,i_{l-1})$, defined in Section \ref{self-affineresults}, which give the probability of having $i_{l}$ as the $\text{l}^{\text{th}}$ coordinate given the previous ones, we can write down an explicit formula for the $\mu $ measure of an approximate cube for any $\omega \in \mathcal{I}^{\mathbb{N}}$ and $r \in (0,1)$:
\begin{equation} \label{approxcubemeasure}
 \mu(\pi(Q(\omega,r)))=\prod^d_{l=1} \prod_{j=0}^{k_l(r)-1}p_l(\sigma^j\omega)
\end{equation}
where $p_l(\omega)=p(i_{1,l}\vert i_{1,1},\ldots,i_{1,l-1})$ and $\sigma: \mathcal{I}^{\mathbb{N}} \to \mathcal{I}^{\mathbb{N}}$ is the left shift.  This formula follows immediately from the definition of $\mu$ and was first observed by Olsen   \cite[(6.2)]{sponges}.

For  $l=1,\ldots, d$, let $p_l^{\text{min}}=\min_{j\in \mathcal{I}} p_l(\mathbf{j})$  and let  $i_l^{\min} \in \mathcal{I}$ be an element achieving this minimum.  If such an element is not unique,  it does not matter which we choose.  Let $t=\sum_{l=1}^d\frac{-\log p_l^{\text{min}}}{\log n_l}$, be the target dimension.  Let $x= \pi(\omega) \in F$ for a (unique) $\omega \in \mathcal{I}^{\mathbb{N}}$ and $0<r<R\le 1$. By the above discussion, and Proposition \ref{ballscubes}, we see that
\[
\frac{\mu \left(\pi \left(Q\left(\omega,\frac{R}{n_1(n_1+ \cdots + n_d)}\right)\right)\right)} {\mu(\pi(Q(\omega,2n_1 r)))} \ \leq \  \frac{\mu(B(x,R))}{\mu(B(x,r))}  \ \leq \ \frac{\mu(\pi(Q(\omega,2n_1 R)))}{\mu \left(\pi \left(Q\left(\omega,\frac{r}{n_1(n_1+ \cdots + n_d)}\right)\right)\right)}
\]
and so to compute  $\r $ it suffices to consider 
\[
\frac{\mu(\pi(Q(\omega,R)))}{\mu(\pi(Q(\omega,r)))} 
\]
directly.  We begin with an upper bound using (\ref{approxcubemeasure}) and for convenience we assume without loss of generality that $r < R/ n_d$, which ensures $k_l(R) < k_l(r)$ for all $l$.  We have
\begin{align*}
\frac{\mu(\pi(Q(\omega,R)))}{\mu(\pi(Q(\omega,r)))} \ = \ \frac{\prod_{l=1}^d\left(\prod_{j=0}^{k_l(R)-1}p_l(\sigma^j \omega) \right)}{\prod_{l=1}^d\left(\prod_{j=0}^{k_l(r)-1}p_l(\sigma^j \omega) \right)}  & = \ \frac{1}{\prod_{l=1}^d\left(\prod_{j=k_l(R)}^{k_l(r)-1}p_l(\sigma^j \omega) \right)} \\
& \le\  \prod_{l=1}^d \frac{1}{\prod_{j=k_l(R)}^{k_l(r)-1}p_l^{\text{min}}} \\
& =\  \prod_{l=1}^d\left( \frac{1}{p_l^{\text{min}}}\right)^{k_l(r)-k_l(R)} \\
& \le\  \prod_{l=1}^d \left( \frac{1}{p_l^{\text{min}}}\right)^{-\log r/\log n_l + \log R/\log n_l + 1}  \\
& \le \ p^{-d} \left( \frac{R}{r} \right)^{t}
\end{align*}
where $p = \min_l p_l^{\text{min}}>0$ is a constant. It follows that $\r \le t$.

For the lower bound, we have to be a little more careful with the relationship between $R$ and $r$.  Again we assume that $r < R/ n_d$, which ensures $k_l(R) < k_l(r)$ for all $l$.  However, for technical reasons we also require $k_l(R) > k_{l+1}(r)$ for all $l = 1, \dots, d-1$.  For this it is sufficient to assume
\[
(n_lR)^{\frac{\log n_{l+1}}{\log n_l}} < r
\]
and fortunately we can choose $r$ satisfying both of these conditions simultaneously.  This is where we use the fact that the $n_l$ are strictly increasing.  Moreover, we can choose a sequence of pairs $(r,R)$ such that $R \to 0$, $R/r \to \infty$ and
\[
(n_lR)^{\frac{\log n_{l+1}}{\log n_l}} < r< R/ n_d.
\]
Let $R,r$ be a pair from this sequence and observe that
\begin{equation} \label{orderingk}
k_d(R)< k_d(r)<k_{d-1}(R)< k_{d-1}(r)< \cdots <k_2(R)<k_2(r)<k_1(R)<k_1(r).
\end{equation}
Let $\omega \in  \mathcal{I}^{\mathbb{N}}$ be chosen such that
\begin{align*} 
\omega&= (i_1,\ldots, i_{k_d(R)}, i_d^{\text{min}},\ldots, i_d^{\text{min}}, i_{k_d(r)+1},\ldots, i_{k_2(R)},\\
&i_2^{\text{min}},\ldots , i_2^{\text{min}}, i_{k_2(r)+1},\ldots, i_{k_1(R)}, i_1^{\text{min}}, \ldots, i_1^{\text{min}}, i_{k_1(r)+1},\ldots).
\end{align*}
where the coordinates not specified as $i_l^{\text{min}}$ (for some $l \in \{1, \dots, d\}$) are arbitrary.  In particular, we insist that the coordinates $i_{k_l(R)+1}, \dots, i_{k_l(r)}$ are all equal to $i_l^{\min}$.   Note that we use (\ref{orderingk}) here.  We have
\begin{align*}
\frac{\mu(\pi(Q(\omega,R)))}{\mu(\pi(Q(\omega,r)))}  \ = \ \frac{\prod_{l=1}^d\left(\prod_{j=0}^{k_l(R)-1}p_l(\sigma^j \omega) \right)}{\prod_{l=1}^d\left(\prod_{j=0}^{k_l(r)-1}p_l(\sigma^j \omega) \right)} & =\ \frac{1}{\prod_{l=1}^d\left(\prod_{j=k_l(R)}^{k_l(r)-1}p_l(\sigma^j \omega) \right)} \\
& = \ \prod_{l=1}^d\left( \frac{1}{p_l^{\text{min}}}\right)^{k_l(r)-k_l(R)}  \\
& \ge\  \prod_{l=1}^d \left( \frac{1}{p_l^{\text{min}}}\right)^{-\log r/\log n_l -1+ \log R/\log n_l }  \\
& \ge \ p^{d} \left( \frac{R}{r} \right)^{t}
\end{align*}
where $p = \min_l p_l^{\text{min}}>0$ is a constant as before.  Since we can choose a sequence of parameters $x = \pi(\omega)$, $r<R$ satisfying the above with $R/r \to \infty$, it follows that $\r \ge t$, completing the proof.

\subsection{Sequences and associated measures}\label{sequenceproof}

We start by explaining our method and then we specialise to the particular cases of  Theorem \ref{sequences} in subsequent subsections.  For convenience, we assume that $x_k \searrow 0$, $x_k-x_{k+1} \searrow 0$ (decreasing gaps) and that $p(k)$ ($k \in \mathbb{N}$) can be extended to a decreasing $L^1$ function $p$ on the whole of $[0,\infty)$.  These conditions are obviously satisfied for the  examples we consider.  Throughout this section we write $f(x)=O(g(x))$ to mean $f(x) \le Cg(x)$ for a constant $C$ independent of $x$.

For  $0<r<1$ and $x\in  \textup{supp}(\mu)$, let $\overline{k}(x,r), \underline{k}(x,r)$   be the unique  integers such that $x_{\overline{k}(x,r)+1}<x-r\le x_{\overline{k}(x,r)}$ and  $x_{\underline{k}(x,r)}\le x+r <x_{\underline{k}(x,r)-1}$ where we adopt the convention that $\overline{k}(x,r) = \infty$ when $x-r \le 0$.    Therefore 
$$
\mu(B(x,r))=\frac{1}{\sum_{n=1}^\infty p(n)} \sum_{i=\underline{k}(x,r)}^{\overline{k}(x,r)} p(i).
$$
Of course there is a possibility that $x_{\overline{k}(x,r)}=x$ when $x\neq 0$ and then $\mu(B(x,r))=\mu(\left\{x\right\})$. Thus, given $0<r<R<1$, we will consider three different cases: 
\begin{enumerate}
\item $x\in \textup{supp}(\mu) \setminus \left\{ 0\right\}$  such that  $x_{\overline{k}(x,r)}=x_{\overline{k}(x,R)}=x$,
\item $x\in \textup{supp}(\mu) \setminus \left\{ 0\right\}$ such that  $x_{\overline{k}(x,R)}<x_{\overline{k}(x,r)}=x$,
\item $x\in \textup{supp}(\mu)$ such that  $x_{\overline{k}(x,R)}\le x_{\overline{k}(x,r)}<x$ or $x=0$.
\end{enumerate}
Case 1 is trivial since
\[
\frac{\mu(B(x,R))}{\mu(B(x,r))}=\frac{\mu(\{x\})}{\mu(\{x\})}=1 \leq \left(\frac{R}{r}\right)^0
\]
and so we omit further discussion of it.   We now consider case 3, which is the most important.  Since $p(x)$ is decreasing, we have
$$p(n+1)\le \int_{n}^{n+1}p(x)dx \le p(n)$$
and therefore
\begin{eqnarray*}
\frac{1}{\sum_{n=1}^\infty p(n)} \int_{\underline{k}(x,r)}^{\overline{k}(x,r)+1} p(x)dx \le  \mu(B(x,r)) &=& \frac{1}{\sum_{n=1}^\infty p(n)}\sum_{i=\underline{k}(x,r)}^{\overline{k}(x,r)} p(i) \\ \\
&\le& \frac{1}{\sum_{n=1}^\infty p(n)} \int_{\underline{k}(x,r)-1}^{\overline{k}(x,r)} p(x)dx .
\end{eqnarray*}
Hence for any $0<r<R<1$ and any $x\in  \textup{supp}(\mu)$ in case 3 we have
\[
\frac{\mu(B(x,R))}{\mu(B(x,r))} \le \frac{\int_{\underline{k}(x,R)-1}^{\overline{k}(x,R)} p(x)dx}{\int_{\underline{k}(x,r)}^{\overline{k}(x,r)+1} p(x)dx} \le \frac{\int_{\underline{k}(x,R)-1}^{\overline{k}(x,R)} p(x)dx}{\int_{\underline{k}(x,r)}^{\overline{k}(x,r)} p(x)dx}.
\]
To simplify this expression for convenience, we assume that $p(x)$ does not decay faster than exponentially in the sense that there is an $\alpha>0$ such that $p(n)/p(n+1) \le \alpha$ for all $n\in \mathbb{N}$. (This is satisfied for all $p(n)$ we consider.) Then
\begin{align*}
\int_{\underline{k}(x,R)-1}^{\overline{k}(x,R)} p(x)dx &=\int_{\underline{k}(x,R)-1}^{\underline{k}(x,R)} p(x)dx + \int_{\underline{k}(x,R)}^{\overline{k}(x,R)} p(x)dx \\
&\le \alpha^2\int_{\underline{k}(x,R)}^{\underline{k}(x,R)+1} p(x)dx + \int_{\underline{k}(x,R)}^{\overline{k}(x,R)} p(x)dx \\
&\le (\alpha^2+1) \int_{\underline{k}(x,R)}^{\overline{k}(x,R)} p(x)dx 
\end{align*}
Thus in case 3 we have 
\[
\frac{\mu(B(x,R))}{\mu(B(x,r))} \le (\alpha^2+1) \frac{\int_{\underline{k}(x,R)}^{\overline{k}(x,R)} p(x)dx}{\int_{\underline{k}(x,r)}^{\overline{k}(x,r)} p(x)dx}.
\]
In case 2, we get 
\[
\frac{\mu(B(x,R))}{\mu(B(x,r))} \le (\alpha^2+1)  \frac{\int_{\underline{k}(x,R)}^{\overline{k}(x,R)} p(x)dx}{\mu(\{x\})}
\]
In what follows we will drop the constants $(\alpha^2+1) $ to simplify notation.  We can now apply these bounds to specific sequences and measures to get upper bounds for the upper regularity dimension.  The lower bounds will be provided by Theorem \ref{relationships}.

\subsubsection{Polynomial-polynomial}

Let $p(n)=n^{-\omega}$ and $x_n=n^{-\lambda}$ with $\lambda>0$ and $\omega>1$.  Let $s=\frac{\omega-1}{\lambda}$ be the target dimension and note that
\begin{align*}
(\overline{k}(x,r)+1)^{-\lambda}&<x-r\le \overline{k}(x,r)^{-\lambda} \\
\underline{k}(x,r)^{-\lambda}&\le x+r< (\underline{k}(x,r)-1)^{-\lambda} 
\end{align*}
with $\overline{k}(x,r) = \infty$ if $x-r\le 0$ and thus (for $x > r$)
\begin{align*}
      (x-r)^{-1/\lambda}-1 &< \overline{k}(x,r) \le (x-r)^{-1/\lambda} \\
      (x+r)^{-1/\lambda} &\le \underline{k}(x,r) < (x+r)^{-1/\lambda}+1.
\end{align*}
We first consider case 3. By our previous calculations, for any $0<r<R<1$ and for any $x\in \textup{supp}(\mu)$ satisfying the conditions required by case 3, we get (up to constants which we ignore)
\begin{eqnarray*}
\frac{\mu(B(x,R))}{\mu(B(x,r))} \le \frac{\int_{\underline{k}(x,R)}^{\overline{k}(x,R)} x^{-\omega}dx}{\int_{\underline{k}(x,r)}^{\overline{k}(x,r)} x^{-\omega}dx} &\le& \frac{\overline{k}(x,R)^{1-\omega}-\underline{k}(x,R)^{1-\omega}}{\overline{k}(x,r)^{1-\omega}-\underline{k}(x,r)^{1-\omega}} \\ 
& \le & \frac{(x+R)^{s}-\max\left\{ x-R,0\right\}^{s}}{(x+r)^{s}-\max\left\{x-r,0\right\}^s}.
\end{eqnarray*}

We are interested in the supremum of this upper bound taken over all $x \in [0,1]$.  It turns out that this can be controlled from above by a constant  multiple of the bound evaluated at $x=0$ or $x=R$.  We demonstrate this by repeated application of Taylor's theorem for   $(1+y)^s$ as a function of $y$ close to 0. In particular,   there exist constants $\varepsilon \in (0,1)$ and $C > 0$ (depending only on $s$) such that for any $y \in [-\varepsilon, \varepsilon]$
\[
1 + s y - C y^2 \le (1+y)^s \le 1 + s y + C y^2,
\]
that is $ (1+y)^s = 1 + s y +O( y^2)$.  We may assume $r<\varepsilon^2 R$ and we consider distinct cases (a), (b) and (c).

\noindent (a) Assume $x\in [0,r]$.  In this case the upper bound is decreasing so a bound obtained at $x = 0$ will be a bound for the whole region. When $x=0$ it follows immediately that
\[
\frac{\mu(B(x,R))}{\mu(B(x,r))} \le \left(\frac{R}{r}\right)^s.
\]
\noindent (b) Assume $r < x < R$. 
 If  $x > r/\varepsilon$, then 
\begin{align*}
\frac{\mu(B(x,R))}{\mu(B(x,r))} &\le \frac{(2R/x)^s}{(1+r/x)^s - (1-r/x)^s} \le \frac{(2R/x)^s}{1+sr/x-1+sr/x + O((r/x)^2)}\\ 
&=  (2R/x)^s O(x/r) \le O\left(\left(\frac{R}{r}\right)^{\max\{s,1\}}\right).    
\end{align*}
If $x < r/\varepsilon$, then 
\[
\frac{\mu(B(x,R))}{\mu(B(x,r))} \le \frac{(2R)^s}{(x+r)^s - (x-r)^s} = O\left(\left(\frac{R}{r}\right)^s\right).
\]
The first bound here is controlled from above by the behaviour at  $x=0$ if $s \geq 1$ and $x=R$ if $s <1$ (see below) whilst the second one is simply controlled by the behaviour at $x=0$.

\noindent (c) Assume  $x \in [R,1]$. If $x \leq R/\varepsilon$, then 
\[
\frac{\mu(B(x,R))}{\mu(B(x,r))} \le \frac{(x+R)^s - (x-R)^s}{(x+r)^s - (x-r)^s} = \frac{(1+R/x)^s - (1-R/x)^s}{(1+r/x)^s - (1-r/x)^s}
\]
and so we use Taylor's Theorem  to obtain
\[
\frac{\mu(B(x,R))}{\mu(B(x,r))} \le \frac{2^s}{1+sr/x + O((r/x)^2)- 1 + sr/x + O((r/x)^2)} \le O\left(\frac{x}{r}\right) = O\left(\frac{R}{r}\right).
\]
If $x > R/\varepsilon$, then  Taylor's theorem can be used on $(1+R/x)^s$ as well yielding
\[
\frac{\mu(B(x,R))}{\mu(B(x,r))} \le \frac{(1+R/x)^s - (1-R/x)^s}{(1+r/x)^s - (1-r/x)^s} \le \frac{1+sR/x - 1 + sR/x + O((R/x)^2)}{1-sr/x -1 + sr/x + O((r/x)^2)} =O\left(\frac{R}{r}\right)
\]
as desired. In particular, the bounds attained here are controlled from above by the  behaviour at $x= R$.   This completes the proof in the original case 3.

We now consider case 2, where we have
\begin{align*}
&\frac{\mu(B(x,R))}{\mu(B(x,r))} \le \frac{(x+R)^{s}-\max\left\{x-R,0\right\}^{s}}{x^{\omega/\lambda}}
\end{align*}
up to a constant which we ignore.  This upper bound is decreasing in $x$ and the case 2 assumption forces a lower bound on $x$ in terms of $r$.  Indeed, let $x=n^{-\lambda}$ and note that, since we are in case 2, we have
\[
n^{-\lambda}-  r > (n+1)^{-\lambda}.
\]
It follows that 
\[
r<  n^{-\lambda}- (n+1)^{-\lambda} =  \frac{-1}{\lambda}\int_{n+1}^n z^{-\lambda-1}dz = \frac{1}{\lambda} n^{-\lambda-1}
\]
and rearranging gives
\[
x= n^{-\lambda} > (\lambda r )^{\lambda/(\lambda +1)}.
\]
Therefore, we have
\[
\frac{\mu(B(x,R))}{\mu(B(x,r))} \le   \frac{((\lambda r )^{\lambda/(\lambda +1)}+R)^{s}-\max\left\{(\lambda r )^{\lambda/(\lambda +1)}-R,0\right\}^{s}}{(\lambda r )^{\omega/(\lambda +1)}}
\]
and we split into two further subcases according to which term dominates in the numerator.
\begin{enumerate}
\item[(i)] If $(\lambda r)^{\lambda/(\lambda+1)}<R$, then
\[
\frac{\mu(B(x,R))}{\mu(B(x,r))} \le \frac{(2R)^{s} }{ (\lambda r )^{\omega/(\lambda +1)}}  \le O \left( (R/r)^s \right)
\]
provided $s \geq \omega/(\lambda +1)$.  If this is not the case, then simple algebra yields $s < 1$.  This, combined with our assumption  $(\lambda r)^{\lambda/(\lambda+1)}<R$, gives
\[
\frac{\mu(B(x,R))}{\mu(B(x,r))} \le \frac{(2R)^{s} }{ (\lambda r )^{\omega/(\lambda +1)}}  \le O \left( R/r \right).
\]
\item[(ii)] If $(\lambda r)^{\lambda/(\lambda+1)} \geq R$, then 
\begin{eqnarray*}
\frac{\mu(B(x,R))}{\mu(B(x,r))} &\le& \frac{((\lambda r )^{\lambda/(\lambda +1)}+R)^{s}- ((\lambda r )^{\lambda/(\lambda +1)}-R)^{s}}{(\lambda r )^{\omega/(\lambda +1)}} 
\end{eqnarray*}
and applying Taylor series estimates in $R/(\lambda r)^{\lambda/(\lambda+1)} \to 0$ similar to above we obtain
\begin{eqnarray*}
\frac{\mu(B(x,R))}{\mu(B(x,r))} &\le& \frac{(\lambda r)^{s\lambda/(\lambda+1)}}{(\lambda r )^{\omega/(\lambda +1)}} \, O \left( R/(\lambda r)^{\lambda/(\lambda+1)} \right) \le O \left( R/r \right).
\end{eqnarray*}
\end{enumerate}

Combining the above estimates gives $\r\le\max \{s,1\}$.  Finally, note that $\overline{\dim}_{\text{loc}}(x,\mu)=0$ when $x\neq 0$, since such $x$ are atoms.  However, when $x=0$, the above estimates  immediately give $\overline{\dim}_{\text{loc}}(0,\mu)= s$ and it is well-known that $\text{supp}(\mu) = 1$, which, combined with Theorem \ref{relationships}, completes the proof of the `polynomial-polynomial' part of Theorem \ref{sequences}.

\subsubsection{Exponential-exponential}

Let $ x_n = \lambda^{n} $ with associated probabilities $p(n)=\omega^{n}$, where $\lambda, \omega \in (0,1)$ and let $s=\frac{\log \omega}{\log \lambda}$ be the target dimension. The situation is much simpler than in the `polynomial-polynomial' case due to the exponential convergence allowing rougher estimates.  In particular, for any $x \in [0,1]$ and $0<r<R$ with $(r/R)$ sufficiently small, we have
\begin{eqnarray*}
\frac{\mu(B(x,R))}{\mu(B(x,r))} \le \frac{\int_{\underline{k}(x,R)-1}^{\overline{k}(x,R)} \omega^x dx}{\int_{\underline{k}(x,r)}^{\overline{k}(x,r)+1} \omega^x dx} = \frac{\omega^{\underline{k}(x,R)-1} - \omega^{\overline{k}(x,R)} }{\omega^{\underline{k}(x,r)}-\omega^{\overline{k}(x,r)+1}  }  &\le&  \frac{\omega^{\underline{k}(x,R)-1}   }{\omega^{\underline{k}(x,r)}(1-\omega)  } \\ \\
&  \le& \frac{1}{\omega^2(1-\omega)} \frac{(x+R)^s}{(x+r)^s} \\ \\
&  \le& O\left(  (R/r)^s \right)
\end{eqnarray*}
which proves that $\r \leq s$.  Again,  it is clear that $\overline{\dim}_{\text{loc}}(x,\mu)=0$ when $x\neq 0$ but $\overline{\dim}_{\text{loc}}(0,\mu)= s$ and it is well-known that $\text{supp}(\mu) = 0$, which resolves the `exponential-exponential' part of  Theorem \ref{sequences}.

\subsubsection{Mixed rates}

We first consider the case where  $x_n=n^{-\lambda}$ with associated probability vector $p(n)=\omega^n$, where $\lambda>0$ and $\omega\in(0,1)$.  Choosing  $x=0$ and $r=R/2$ we get
\[
 \frac{\mu(B(0,R))}{\mu(B(0,R/2))}   \geq  \frac{\int_{\underline{k}(x,R)}^\infty \omega^x dx}{\int_{\underline{k}(x,R/2)-1}^\infty \omega^x dx} \geq   \omega^{R^{-1/\lambda}-(R/2)^{-1/\lambda}} \to \infty
\]
as $R \to 0$, which proves that $\mu$ is not doubling, as required.
\\ \\
We now consider the opposite case, where $x_n=\lambda^{n}$ with associated probability vector $p(n)=n^{-\omega}$, where $\lambda\in (0,1)$ and $\omega > 1$. Curiously, and in contrast to the previous case,  the measure is `very doubling' at 0. As such, to demonstrate that the measure is non-doubling, we choose $x=R$ and $r=R/2$.  Our previous estimates yield that for sufficiently small $R>0$ and up to a constant that we ignore
\begin{eqnarray*}
\frac{\mu(B(R,R))}{\mu(B(R,R/2))}   \geq  \frac{(-\log 2R)^{1-\omega}}{(-\log \frac{3R}{2})^{1-\omega} - (-\log\frac{R}{2} )^{1-\omega}}   \to \infty
\end{eqnarray*}
as $R \to 0$, proving that  $\mu$ is non-doubling.

\section*{Acknowledgements}

The authors thank Xiong Jin for originally suggesting that we investigate the upper regularity dimension and Han Yu for many fruitful conversations. We also thank Antti K\"aenm\"aki for providing useful references, Haipeng Chen for helpful comments, and an anonymous referee for their corrections. J.M.F. was financially supported by a \emph{Leverhulme Trust Research Fellowship} (RF-2016-500) and D.C.H. was financially supported by an EPSRC Doctoral Training Grant (EP/N509759/1).

\end{document}